\documentclass[11pt]{article}
\usepackage{amsmath,amssymb,amsfonts,amsthm}
\usepackage{float}
\usepackage{graphicx}
\numberwithin{equation}{section}
\newtheorem{theorem}{Theorem}[section]
\newtheorem{definition}{Definition}[section]
\newtheorem{lemma}[theorem]{Lemma}
\newtheorem{proposition}[theorem]{Proposition}
\newtheorem{corollary}[theorem]{Corollary}

\begin{document}
\begin{center}
{\Large{\textbf{{A Study on Linear Jaco Graphs}}}} 
\end{center}
\vspace{0.5cm}
\large{
\centerline{(Johan Kok, Susanth C, Sunny Joseph Kalayathankal)\footnote {\textbf {Affiliation of authors:}\\
\noindent Johan Kok (Tshwane Metropolitan Police Department), City of Tshwane, Republic of South Africa\\
e-mail: kokkiek2@tshwane.gov.za\\ \\
\noindent Susanth C (Department of Mathematics, Vidya Academy of Science and Technology), Thalakkottukara, Thrissur-680501, Republic of India\\
e-mail: susanth\_c@yahoo.com\\ \\
\noindent Sunny Joseph Kalayathankal (Department of Mathematics, Kuriakose Elias College), Mannanam, Kottayam- 686561, Kerala, Republic of India\\
e-mail: sunnyjoseph2014@yahoo.com}}
\vspace{0.5cm}
\begin{abstract}
\noindent  We introduce the concept of a family of finite directed graphs (\emph{positive integer order,} $f(x) = mx + c; x,m \in \Bbb  N$ and $c \in \Bbb N_0)$ which are directed graphs derived from an infinite directed graph called the $f(x)$-root digraph. The $f(x)$-root digraph has four fundamental properties which are; $V(J_\infty(f(x))) =\{v_i:i \in \Bbb N\}$ and, if $v_j$ is the head of an arc then the tail is always a vertex $v_i, i<j$ and, if $v_k$ for smallest $k \in \Bbb N$ is a tail vertex then all vertices $v_ \ell, k< \ell<j$ are tails of arcs to $v_j$ and finally, the degree of vertex $v_k$ is $d(v_k) = mk +c.$ The family of finite directed graphs are those limited to $n \in \Bbb N$ vertices by lobbing off all vertices (and corresponding arcs) $v_t, t > n.$ Hence, trivially we have $d(v_i) \leq mi + c$ for $i \in \Bbb N.$ It is meant to be an \emph {introductory paper} to encourage further research.
\end{abstract}
\noindent {\footnotesize \textbf{AMS Classification Numbers:} 05C07, 05C38, 05C78} 
\section{Introduction}
\noindent For general notation and concepts in graph theory, we refer to [2, 3, 4, 8]. All graphs mentioned in this paper are simple, connected finite and directed graphs unless mentioned otherwise. We introduce the concept of a family of finite linear Jaco graphs (\emph{positive integer order,} $f(x) = mx + c; x,m \in \Bbb N$ and $c \in \Bbb N_0)$ which are directed graphs derived from the infinite linear Jaco graph (\emph{positive integer order,} $f(x) = mx + c; x,m \in \Bbb N$ and $c \in \Bbb N_0)$, called the $f(x)$-root Jaco graph. In this study we generalise a number of results found in [6, 7] because the aforesaid studies relate to special cases of linear Jaco graphs.
\begin{definition}
Let $f(x) = mx + c; x,m \in \Bbb N,$ $c\in \Bbb N_0.$ The family of infinite linear Jaco graphs denoted by $\{J_\infty(f(x)):f(x) = mx + c; x,m \in \Bbb N$ and $c \in \Bbb N_0\}$ is defined by $V(J_\infty(f(x))) = \{v_i: i \in \Bbb N\}$, $A(J_\infty(f(x))) \subseteq \{(v_i, v_j): i, j \in \Bbb N, i< j\}$ and $(v_i,v_ j) \in A(J_\infty(f(x)))$ if and only if $(f(i) + i) - d^-(v_i) \geq j.$
\end{definition}
From definition 1.1 we note the $f(x)$-root Jaco graph has four fundamental properties which are;\\\\
(a) $V(J_\infty(f(x))) =\{v_i:i \in \Bbb N\}$ and,\\
(b) if $v_j$ is the head of an arc then the tail is always a vertex $v_i, i<j$ and,\\
(c) if $v_k$ for smallest $k \in \Bbb N$ is a tail vertex then all vertices $v_ \ell, k< \ell<j$ are tails of arcs to $v_j$ and finally,\\
(d) the degree of vertex $k$ is $d(v_k) = mk + c.$ 
\begin{definition}
\noindent For $f(x) = mx + c; x,m \in \Bbb N$ and $c \in \Bbb N_0,$ we define the series $(c_{f(x),n})_{n \in \Bbb N_0}$ by
\begin{center} $c_{f(x), 0} = 0,$ $c_{f(x),1} = 1,$ $c_{f(x),n\geq 2}= \min \{ k < n: mk + c_{f(x), k} \geq n \}$.
\end{center}
\end{definition}
\noindent The connection between the $f(x)$-root digraph $J_\infty(f(x))$ and the series $(c_{f(x),n})_{n\in \Bbb N_0}$ is explained by the following lemma.
\begin{lemma} 
Consider the Jaco graph $J_\infty(f(x))$ and let $n \in \Bbb N$ then the following hold:\\
(a) $d^+(v_n) + d^-(v_n) = f(n).$\\
(b) $d^-(v_{n+1}) \in \{ d^- (v_n),\, d^-(v_n) + 1 \}.$\\
(c) If $(v_i, v_k) \in A(J_\infty(f(x)))$ and $i < j < k,$ then $(v_j, v_k) \in A(J_\infty(f(x))).$\\
(d)  $d^+ (v_n) = ((m-1)n + c) + c_{f(x),n},$ $n \geq 2.$
\end{lemma}
\begin{proof} 
(a) Invoking definition 1.1 and since $d^+(v_n) = (f(n) + n) - n - d^-(v_n),$ the result is obvious.\\
(b), (c) We prove (b) and (c) simultaneously through induction on $n.$ First of all, $d^-(v_1) = 0$ implying $d^-(v_2) = 1 = d^-(v_1) + 1.$\\ \\
Let $n \geq 2$ and assume result (b) holds for $m \leq n$ and (c) holds for $m \leq n-1.$ In particular, $ d^-(v_n) > 0.$ Let $\ell <n$ be minimal with $(v_\ell, v_n) \in A(J_\infty(f(x))),$ so $(f(\ell) + \ell) - d^-(v_\ell) \geq n.$ Let $\ell < j < n.$ By induction, we have
$d^-(v_\ell) \leq d^-(v_j) \leq d^-(v_\ell) + j - \ell$ and $(f(j)+j) - d^-(v_j) \geq n.$ Hence, and by choice of $\ell,$ we have $(v_k, v_n) \in A(J_\infty(f(x)))$ if and only if $\ell \leq k < n,$ hence  result $(c)$ is valid for $n,$ while $d^-(v_n) = n - \ell$ and $d^+(v_n) = (f(n) + n) - (n - \ell) = f(n) + \ell.$\\ \\
If $(f(\ell) + \ell) - d^-(v_\ell) \geq n + 1,$ then $\ell$ is minimal with $(f(\ell) + \ell) - d^-(v_\ell) \geq n + 1.$ If $(f(\ell) + \ell) - d^-(v_\ell) = n,$ we still, as $\ell + 1 \leq n,$ have $d^- (v_{\ell +1}) \in  \{ d^- (v_\ell),\, d^-(v_{\ell) + 1} \}$ and $(f(\ell+1) + \ell + 1) - d^-(v_{\ell +1}) \geq n+1.$  Either way, $(v_{\ell + 1}, v_{n+1}) \in A(J_\infty(f(x))).$ If $\ell + 1 < j < n,$ then induction yields $d^-(v_{\ell +1}) \leq d^-(v_j) \leq d^- (v_{\ell +1}) + (j - \ell -1) $ and $(f(j) + j) - d^-(v_j)   \geq n + 1.$ As $d^-(v_n) \leq n-1, (v_n, v_{n+1}) \in A(J_\infty(f(x))),$ so $(v_k, v_{n+1}) \in A(J_\infty(f(x)))$ whenever $\ell +1 \leq k \leq n.$ Depending on whether $(v_\ell, v_{n+1}) \in A(J_\infty(f(x)))$ or not, we obtain $d^-(v_{n+1}) = n+1 - \ell = (n-\ell) + 1 = d^- (v_\ell) + 1$ or $d^-(v_{n+1}) = n+1 - (\ell+1) = d^- (v_n).$
(d) Let $n \geq 3,$ and, as before, choose $\ell$ minimal with $(v_\ell, v_n) \in A(J_\infty(a)).$ We prove the result by induction on $n$ and apply arguments very similar to the ones already used. First of all, $d^-(v_1) = 0,$ and $d^+ (v_1) = f(1) = ((m-1)1 + c) + c_{f(x),1}.$ Now let $n > 1,$ and, as before, choose $\ell$  minimal such that $(v_\ell, v_n) \in A(J_\infty(a)).$ By $(a)$ and $(c),$ $d^- (v_n) = n - \ell,$ and $d^+(v_n) = (mn + c) - n + \ell = ((m-1)n + c) + \ell.$ Induction yields that $d^+(v_k) = d^+ (v_k) = ((m-1)k + c) + c_{f(x),k},$ whenever $k < n.$  The definition of $\ell$ says that  $\ell$ is minimal with $\ell + d^+(v_\ell) = ((m-1)\ell + c) + c_{f(x),\ell} \geq n,$ which means that $\ell  = c_{f(x),n}.$
\end{proof}
\begin{corollary}
Note that $(a)$ and $(b)$ of Lemma 1.1 entail that $d^-(v_{n+1}) = (n+1) - c_{f(x), n+1} \in \{ n - c_{f(x),n}, n - c_{f(x), n} + 1\}$ and that $(d)$ then implies that the series  $(c_{f(x),n})$ are well-defined and ascending, more specifically, $c_{f(x),n+1} \in \{ c_{f(x),n}, c_{f(x),n} + 1 \},$ ($n \in \Bbb N_0$).
\end{corollary}
Consider the function $g(x) = m_1x + c_1; x,m_1 \in \Bbb N,$ $c_1\in \Bbb N_0.$ In respect of $f(x)$ and $g(x)$ we have the following proposition.
\begin{proposition}
Let $k \in \Bbb N,$ and $0 \leq g(x) < f(x) .$ Then $c_{f(x),mk + c_{f(x),k} - g(x)} = k.$
\end{proposition}
\begin{proof}
Let $mk + c_{f(x),k} - g(x) = \ell.$ Certainly, $mk + c_{f(x), k} \geq \ell$ hence, $c_\ell \leq k.$  On the other hand,
$((m-1)k + c) + c_{f(x), k} = (mk + c) + c_{f(x),k} - k < \ell,$ so Corollary 1.2 says $((m-1)k + c) + c_{f(x), k-1} < \ell$ and $c_\ell = k.$
\end{proof}
\section{Finite Linear Jaco Graphs $\{J_n(f(x)):f(x) = mx + c; x,m\in \Bbb N$ and $c \in \Bbb N_0\}$} 
\noindent The family of finite linear Jaco graphs are those limited to $n \in \Bbb N$ vertices by lobbing off all vertices (and corresponding arcs) $v_t, t > n.$ Hence, trivially we have $d(v_i) \leq f(i)$ for $i \in \Bbb N.$
\begin{definition}
The family of finite linear Jaco graphs denoted by $\{J_n(f(x)):f(x) = mx + c; x,m \in \Bbb N$ and $c\in \Bbb N_0\}$ is the defined by $V(J_n(f(x))) = \{v_i: i \in \Bbb N, i \leq n \}$, $A(J_n(f(x))) \subseteq \{(v_i, v_j)| i,j \in \Bbb N, i< j \leq n\}$ and $(v_i,v_ j) \in A(J_n(f(x)))$ if and only if $(f(i) + i) - d^-(v_i) \geq j.$
\end{definition} 
\begin{definition}
Vertices with degree $\Delta (J_n(f(x)))$ is called Jaconian vertices and the set of vertices with maximum degree is called the Jaconian set of the linear Jaco graph $J_n(f(x)),$ and denoted, $\Bbb{J}(J_n(f(x)))$ or, $\Bbb{J}_n(f(x))$ for brevity.
\end{definition}
\begin{definition}
The lowest numbered (indexed) Jaconian vertex is called the prime Jaconian vertex of a Jaco graph.
\end{definition}
\begin{definition}
If $v_i$ is the prime Jaconian vertex, the complete subgraph on vertices $v_{i+1}, v_{i+2},\\ \cdots, v_n$ is called the Hope subgraph of a linear Jaco graph and denoted, $\Bbb{H}(J_n(f(x)))$ or, $\Bbb{H}_n(f(x))$ for brevity.
\end{definition}
\noindent {\bf Property 1:} From the definition of a linear Jaco graph $J_n(f(x)),$ it follows that, if for the prime Jaconian vertex $v_i,$ we have $d(v_i)= f(i)$ then in the underlying Jaco graph denoted $J_n^*(f(x))$ we have $ d(v_m)= f(m)$ for all $m\in \{1,2,3,\cdots,i\}.$\\ \\
{\bf Property 2:} From the definition of a linear Jaco graph $J_n(f(x)),$ it follows that $\Delta (J_k(f(x)))\leq \Delta (J_n(f(x)))$ for all $k\leq n.$\\ \\
{\bf Property 3:} From the definition of a linear Jaco graph $J_n(f(x)),$ it follows that the lowest degree attained by all Jaco graphs is $0 \leq \delta(J_n(f(x))) \leq f(1).$\\ \\
{\bf Property 4:} The $d^-(v_k)$ for any vertex $v_k$ of a linear Jaco graph $J_n(f(x)), n\geq k$ is equal to $d(v_k)$ in the underlying Jaco graph $J^*_k(f(x)).$\\\\
Note that henceforth, when the context is clear we interchangeably refer to a linear Jaco graph $J_n(f(x)),$ (directed) and the underlying Jaco graph $J^*_n(f(x)),$ (undirected) as a \emph{Jaco graph}. Similar, when the context is clear we refer to either \emph{arcs} or \emph{edges}. 
\newpage
\textbf{Illustration 1.} Figure 1 depicts the Jaco graph $J_{11}(f(x)),$ $ f(x) = 2x + 1.$ We note that vertex $v_4$ is the prime Jaconian vertex and $\Bbb J_{11}(f(x)) =\{v_4, v_5, v_6\}.$ 

\begin{figure}
\centering
\includegraphics[width=0.7\linewidth]{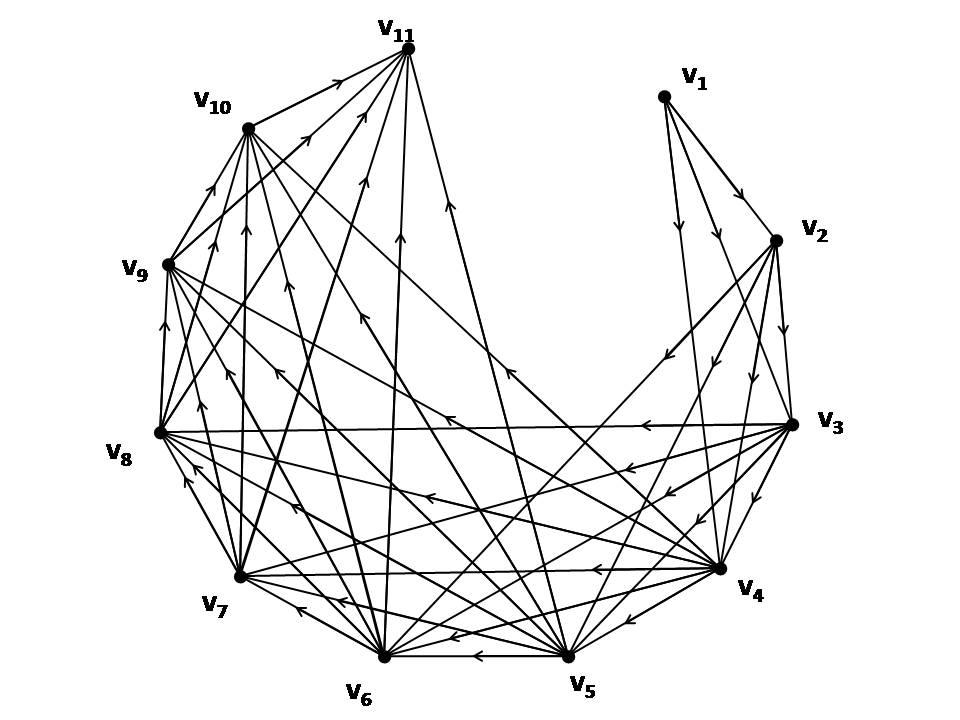}
\caption{}
\label{fig:Fig1}
\end{figure}

\begin{lemma}
For the Jaco graphs $J_i(f(x)), i\in \{1,2,3,...,f(1) + 1\}$ we have $\Delta (J_ i(f(x))) = i-1$ and $\Bbb {J}(J_i(f(x))) = \{v_k: 1\leq k \leq i\} =V(J_i(f(x))).$
\end{lemma}
\begin{proof}
If $t =f(1)+1$ then, $((f(1)+1) +1) - d^-(v_1) > (f(1)+1)$ so from definition 2.1 it follows that the arcs $(v_1,v_i), i= 2,3, ..., (f(1)+1)$ exist. It then follows that all arcs $(v_i,v_j), i<j$ exist. So the underlying graph $J^*_{f(1)+1}(f(x))$ is the complete graph $K_{f(1)+1}.$ Since $\Delta (J_{f(1)+1}(f(x))) = (f(1)+1) - 1 = f(1)$ and we have $d(v_i) = f(1)$ for all vertices in $K_{f(1)+1},$ it follows that $\Delta (J_{f(1)+1}(f(x))) = f(1)$ and $\Bbb {J}(J_{f(1)+1}(f(x))) = \{v_k: 1\leq k \leq f(1)\} =V(J_{f(1)+1}(f(x))).$  The result follows similarly for $t < f(1)+1.$
\end{proof} 
\begin{proposition}
For the Jaco graphs $J_{f(i)}(f(x)),$ $1 \leq i \leq f(1)$ the prime Jaconian vertex is the vertex $v_i$ and $\Bbb J_{f(i)}(f(x)) =\{v_\ell: i \leq \ell \leq f(1) + 1\}.$
\end{proposition}
\begin{proof}
Consider the Jaco graph $J_{f(1)}(f(x)).$ It follows that the induced subgraph $ J' = \langle v_i; 1 \leq i \leq v_{f(1) + 1}\rangle$ is the complete graph $K_{f(1) + 1},$ hence $d_{J'}(v_i) = f(1),$ $1 \leq i \leq f(1) + 1.$ It follows immediately that $v_1$ is the prime Jaconian vertex of $J'$ and $\Bbb J(J') = \{v_\ell: 1\leq \ell \leq f(1) + 1\}.$\\\\
Also $f(i) > f(1)$, $2 \leq i \leq f(1)$ so besides the arcs in $J'$, $v_i$ has additional arcs $(v_i, v_s); f(1) + 2 \leq s \leq f(i) + 1$ in $J_{f(i) + 1}(f(x)).$ This amounts to $(f(i) - f(1)) -1$ additional arcs $\forall v_\ell,$ $ i \leq \ell \leq f(1)$ in $J_{f(i) + 1}(f(x)).$ Clearly $i$ is minimal such that arc $(v_i, v_{f(i) +1})$ exists so $v_i$ is the prime Jaconian vertex of the Jaco graph $J_{f(i) + 1}(f(x)).$\\\\
Because $d_{J'}((v_\ell) = f(1),$ $i\leq \ell \leq f(1) + 1$ it follows immediately that $d_{J'}(v_\ell) + (f(i) - f(1)) -1 = d_{J'}(v_i) + (f(i) - f(1)) -1$ in $J_{f(i) + 1}(f(x)).$ Hence, $\Bbb J_{f(i)}(f(x)) =\{v_\ell: i \leq \ell \leq f(1) + 1\}.$
\end{proof} 
\begin{lemma}
If in a Jaco graph $J_n(f(x)),$ and for smallest $i,$ we have $d(v_i) = f(i)$ and the arc $(v_i, v_n)$ is defined, then $v_i$ is the prime Jaconian vertex of $J_n(f(x)).$
\end{lemma}
\begin{proof}
If in the construction of a Jaco graph $J_n(f(x)),$ and for smallest $i,$ the arc $(v_i, v_n)$ is defined, we have that in the underlying graph $J^*_n(f(x))$, $d(v_i) \leq f(i)$ and $d(v_j) \leq d(v_i)$ for all $j > i$. So it follows that $d(v_i) = \Delta(J_n(f(i)))$ so $v_i$ is the prime Jaconian vertex of $J_n(f(x)).$ 
\end{proof}
\begin{lemma}
For all Jaco graphs $J_n(f(x)),$ $n\geq2$ and, $v_i, v_{i-1}\in V(J_n(f(x)))$ we have in the Jaco graph $J^*_n(f(x)),$ that $|(d(v_i) - d(v_{i-1})|\leq  m.$ 
\end{lemma}
\begin{proof}
Clearly, $max|(d(v_i) - d(v_{i-1})| = f(i) - f(i-1) = mi +c - (m(i-1) + c) = m$ hence, $|(d(v_i) - d(v_{i-1})|\leq  m.$ 
\end{proof}
\begin{proposition}
For the Jaco graph $J_{f(i)}(f(x)),$ $i \geq f(1) + 1$ the prime Jaconian vertex is a vertex $v_j$, $j \leq i.$
\end{proposition}
\begin{proof}
We have that $f(i) - f(i-1) \geq 1.$ From Lemma 1.1(b) it follows that $d^-(v_i) = d^-(v_{i-1})$ or $ d^-(v_{i-1}) + 1.$ \\
Case 1: Let $f(i) - f(i-1) = 1.$\\
Subcase 1.1: Let $d^-(v_i) = d^-(v_{i-1}).$ It follows that $d_{J_{f(i)}}(v_{i-1}) = d_{J_{f(i)}}(v_i)+1$ hence $v_{i-1}$ is the prime Jaconian vertex. Let $j = i-1$, so $j < i.$\\
Subcase 1.2:  Let $d^-(v_i) = d^-(v_{i-1}) + 1.$ It follows that $d_{J_{f(i)}}(v_{i-1}) = d_{J_{f(i)}}(v_i)$ hence $v_{i-1}$ is the prime Jaconian vertex. Let $j = i-1$, so $j < i.$\\
Case 1: Let $f(i) - f(i-1) = 2.$\\
Subcase 2.1: Let $d^-(v_i) = d^-(v_{i-1}).$ It follows that $d_{J_{f(i)}}(v_{i-1}) = d_{J_{f(i)}}(v_i)$ hence $v_{i-1}$ is the prime Jaconian vertex. Let $j = i-1$, so $j < i.$\\
Subcase 1.2:  Let $d^-(v_i) = d^-(v_{i-1}) + 1.$ It follows that $d_{J_{f(i)}}(v_i) = d_{J_{f(i)}}(v_{i-1}) + 1$ hence $v_i$ is the prime Jaconian vertex. So equality holds.\\
By Similar reasoning it follows that for all $f(i) - f(i-1) \geq 3$ vetex $v_i$ is the prime Jaconian vertex hence, the result is settled.
\end{proof}
\noindent Note that the value of $\Delta (J_n(f(x)))$ might repeat itself as $n$ increases to $n+1$ and on an increase we always obtain, $\Delta (J_{n+1}(f(x))) = \Delta (J_n(f(x)))+ m$.
\begin{theorem}
The Jaco graph $J_k(f(x))$, $k= f((m+ c)+ 1) + 1$ is the smallest Jaco graph in $\{J_n(f(x)): n \in \Bbb {N}\}$ which has $\Delta (J_k(f(x))) = f((m+c) + 1)$ with unique Jaconian vertex $v_{(m+c) + 1}.$ 
\end{theorem}
\begin{proof}
We will prove the result through \emph{imbedded} induction.\\
Let $c = 0$ so consider the linear function $ f(x) = mx.$ For $m = 1$ let $f'(x) = x.$ The graph $J_3(f'(x))$ is clearly the smallest Jaco graph in $\{J_n(f'(x)): n \in \Bbb N\}$ for which $\Delta (J_3(f'(x))) =(1 + 0) + 1 = 2$ and $\Bbb J(J_3(f'(x))) = \{v_{((1 + 0)+1)}\} = \{v_2\}.$ So the result holds for $m = 1.$\\ \\
Assume the result holds for $m= \ell,$ and let $f''(x) = \ell x.$ So for the smallest Jaco graph $J_l(f''(x)), l = \ell(\ell + 1) + 1$ we have that $\Delta (J_l(f''(x))) = \ell(\ell+1)$ and $\Bbb J(J_l(f''(x))) = \{v_{(\ell+1)}\}.$ \\ \\
Now consider the linear function $f'''(x) = (\ell + 1)x.$ In the Jaco graph $J_l(f''(x))$ the vertex $v_{(\ell+1)+1} = v_{(\ell+2)}$ has $d(v_{(\ell+2)}) = d(v_{(\ell+1)}) - 1.$ So in constructing the Jaco graph $J_l(f'''(x)),$ amongst others the arc $(v_1, v_{(m+2)})$ is linked.  So at least $v_{(\ell+1)}, v_{(\ell+2)} \in \Bbb J(J_l(f'''(x))).$ So $d(v_{(\ell+2)}) = \ell(\ell+1).$ If follows that the minimum number of additional vertices (smallest Jaco graph) say $t,$ to be added to $J_l(f'''(x))$ to obtain $d(v_{(\ell+2)}) = (\ell+1)(\ell+2)$ and $\Bbb J(J_{(l+t)}(f'''(x))) = \{v_{(\ell+2)}\}$ in $J_{(l+t)}(f'''(x))$ is given by $t = (\ell+1)(\ell+2) - \ell(\ell+1) = 2(\ell+1).$ 
The number of vertices of $J_l(f''(x))$ is given by $\ell(\ell+1) +1.$ Now $l + t = (\ell(\ell+1) + 1) +2(\ell+1) = (\ell+1)(\ell+2) + 1.$ \\ \\
Clearly at least $v_{(\ell+2)} \in \Bbb J(J_k(f'''(x))), k= l+t,$ and $v_{(\ell+1)} \notin \Bbb J(J_k(f'''(x)))$ since $\ell(\ell+1) < (\ell+1)(\ell+2).$ In the construction of the Jaco graph $J_l(f'''(x)),$ the arc $(v_1, v_{(\ell+3)})$ was not linked so $d(v_{(\ell+3)}) < d(v_{(\ell+2)})$ in $J_l(f'''(x)).$ So it follows that $d(v_{(\ell+3)}) < d(v_{(\ell+2)})$ in $J_k(f'''(x)), k = l + t.$ The latter implies that $J(J_k(f'''(x)))$ is the smallest such Jaco graph and $\Bbb J(J_k(f'''(x))) = \{v_{(\ell+2)}\}.$\\ \\
Hence the result holds for $m = \ell + 1$ implying it holds in general for $c = 0.$\\\\
Now let $c = 1,$ and $f(x) = mx + 1.$ For $m=1$ let $f'(x) = x + 1.$ Clearly $J_5(f'(x))$ is the smallest Jaco gaph in $\{J_n(f'(x)): n \in \Bbb N\}$ for which $\Delta (J_5(f'(x))) = f'(2) + 1 = 4$ and $\Bbb J(J_5(f'(x))) = \{v_{(2+1)}\} = \{v_3\}.$ So the result holds for $m = 1.$\\ \\
Assume it holds for $m = \ell,$ and let $f''(x) = \ell x + 1.$ So for the smallest Jaco graph $J_l(f''(x)), l = \ell^2 + 2\ell + 1$ we have that $\Delta (J_l(f''(x))) = \ell(\ell+2)$ and $\Bbb J(J_l(f''(x))) = \{v_{(\ell+2)}\}.$ \\ \\ 
Now consider the function $f'''(x) = (\ell+1)x + 1.$ In the Jaco graph $J_l(f''(x))$ the vertex $v_{(\ell+2)+1} = v_{\ell+3}$ has $d(v_{\ell+3}) = d(v_{\ell+2}) - 1.$ So in the construction of the Jaco graph $J_l(f'''(x)),$ amogst others the arc $(v_1, v_{\ell+3})$ is linked. So at least $v_{(\ell+2)}, v_{(\ell+3)} \in \Bbb J(J_l(f'''(x))).$ So $d(v_{(\ell+3)}) = \ell(\ell+2).$ If follows that the minimum number of additional vertices (smallest Jaco graph) say $t,$ to be added to $J_l(f'''(x))$ to obtain $d(v_{(\ell+3)}) = (\ell+2)(\ell+3)$ and $\Bbb J(J_{(l+t)}(f'''(x))) = \{v_{(\ell+3)}\}$ in $J_{(l+t)}(f'''(x))$ is given by $t = (\ell+2)(\ell+3) - \ell(\ell+2) = 2(\ell+2).$ 
The number of vertices of $J_l(f''(x))$ is given by $(\ell+1)(\ell+2) +1.$ Now $l + t = ((\ell +1)(\ell+2) + 1) +2((\ell+1)+1) = (\ell+2)(\ell+3) + 1.$ \\ \\
Clearly at least $v_{(\ell+3)} \in \Bbb J(J_k(f'''(x))), k= l+t,$ and $v_{(\ell+2)} \notin \Bbb J(J_k(f'''(x)))$ since $(\ell+1)(\ell+2) < (\ell+2)(\ell+3).$ In the construction of the Jaco graph $J_l(f'''(x)),$ the arc $(v_1, v_{(\ell+4)})$ was not linked so $d(v_{(\ell+4)}) < d(v_{(\ell+3)})$ in $J_l(f'''(x)).$ So it follows that $d(v_{(\ell+4)}) < d(v_{(\ell+3)})$ in $J_k(f'''(x)), k = l + t.$ The latter implies that $J(J_k(f'''(x)))$ is the smallest such Jaco graph and $\Bbb J(J_k(f'''(x))) = \{v_{(\ell+3)}\}.$\\ \\
Hence the result holds for $m = \ell + 1$ implying it holds in general for $c = 1.$\\\\
Now assume the result holds for $1 \leq m \leq \ell_1$ and $0 \leq c \leq c_\ell$. This means that for $f'(x) = \ell_1 x + c_\ell$ the smallest Jaco graph $J_l(f'(x))$ has $\Delta (J_l(f'(x))) = f'((\ell_1 + c_\ell) +1) = \ell_1(\ell_1 + c_\ell) + 1) + c_\ell.$ and $\Bbb J)J_l(f'(x))) = \{v_{(\ell_1 + c_\ell) + 1}\}.$ Now consider the linear function $f''(x) = (\ell_1 + 1)x + (c_\ell + 1).$\\\\
In the Jaco graph $J_l(f'(x))$ the vertex $v_{(\ell+ c_\ell+2)+1} = v_{\ell+ c_\ell+3}$ has $d(v_{\ell+ c_\ell+3}) = d(v_{\ell + c_\ell +2}) - 1.$ So in the construction of the Jaco graph $J_l(f''(x)),$ amogst others the arc $(v_1, v_{\ell + c_\ell+3})$ is linked. So at least $v_{(\ell + c_\ell+2)}, v_{(\ell+ c_\ell+3)} \in \Bbb J(J_l(f''(x))).$ So $d(v_{(\ell+ c_\ell+3)}) = (\ell+c_\ell)(\ell + c_\ell+2).$ If follows that the minimum number of additional vertices (smallest Jaco graph) say $t,$ to be added to $J_l(f''(x))$ to obtain $d(v_{(\ell + c_\ell+3)}) = (\ell + c_\ell+2)(\ell+ c_\ell+3)$ and $\Bbb J(J_{(l+t)}(f''(x))) = \{v_{(\ell+c_\ell+3)}\}$ in $J_{(l+t)}(f''(x))$ is given by $t = (\ell+ c_\ell+2)(\ell+ c_\ell+3) - (\ell + c_\ell)(\ell + c_\ell+2) = 2(\ell + c_\ell+2).$ 
The number of vertices of $J_l(f'(x))$ is given by $(\ell + c_\ell+1)(\ell + c_\ell+2) +1.$ Now $l + t = ((\ell + c_\ell+1)(\ell + c_\ell+2) + 1) +2((\ell + c_\ell+1)+1) = (\ell + c_\ell+2)(\ell + c_\ell+3) + 1.$ \\ \\
Clearly at least $v_{(\ell + c_\ell+3)} \in \Bbb J(J_k(f''(x))), k= l+t,$ and $v_{(\ell + c_\ell+2)} \notin \Bbb J(J_k(f''(x)))$ since $(\ell + c_\ell+1)(\ell + c_\ell+2) < (\ell +c_\ell+2)(\ell + c_\ell+3).$ In the construction of the Jaco graph $J_l(f''(x)),$ the arc $(v_1, v_{(\ell = c_\ell+4)})$ was not linked so $d(v_{(\ell + c_\ell+4)}) < d(v_{(\ell + c_\ell+3)})$ in $J_l(f''(x)).$ So it follows that $d(v_{(\ell + c_\ell+4)}) < d(v_{(\ell + c_\ell+3)})$ in $J_k(f''(x)), k = l + t.$ The latter implies that $J(J_k(f''(x)))$ is the smallest such Jaco graph and $\Bbb J(J_k(f''(x))) = \{v_{(\ell + c_\ell+3)}\}.$\\ \\
Hence the result holds for $m = \ell_1 + 1$ and $c_\ell + 1$ implying it holds in general for $f(x) = mx + c.$
\end{proof}
\section{Number of Arcs of the Finite Jaco Graphs $\{J_n(f(x)):f(x) = mx + c; x,m\in \Bbb N$ and $c \in \Bbb N_0\}$}
\noindent Note,  in [6,7] it was thought that Theorem 3.7 combined with Binet's formula perhaps amounts to a closed formula for $d^+(v_n).$ It is hoped that as a generalised case, a closed formula can be found for the number of arcs of a finite Jaco graph $J_n(f(x)).$ Theorem 3.8 provides a result for a subset of the family of Jaco graphs $J_n(f(x)).$ The algorithms discussed in Ahlbach et al.[1] suggest that a general closed formula might not be found. First we present results for $c=0$ hence for $f(x) = mx.$
\begin{proposition}
The number of arcs of a Jaco graph $J_{\ell}(mx) = \frac{1}{2}\ell(\ell -1)$ if $\ell \leq m+1.$
\end{proposition}
\begin{proof}
If $\ell = m+1$ then, $((m+1) +1).1 - d^-(v_1) > (m+1)$ so from definition 2.1 it follows that the arcs $(v_1,v_i), i= 2,3, ..., (m+1)$ exist. It then follows that all arcs $(v_i,v_j), i<j$ exist. So the underlying graph $J^*_{m+1}(m)$ is the complete graph $K_{m+1}$ hence, $\epsilon (J_{m+1}(m)) =  \frac{1}{2}m(m+1) = \frac{1}{2}\ell(\ell-1).$  The result follows similarly for $\ell < m+1.$
\end{proof}
\begin{theorem}
If for the Jaco graph $J_n(mx),$ we have $\Delta (J_n(mx)) = k,$ then  $\epsilon(J_n(mx)) = \epsilon(\Bbb{H}(J_n(mx))) + \sum \limits_{i=1}^{k}d^+(v_i). $
\end{theorem}
\begin{proof} 
For $n=1,$ $d^+(v_1) = 0$ and $J_1(mx)$ is the arcless graph on vertex $v_1$ whilst $\Bbb{H}(J_1(mx)),$ is an empty graph so $A(\Bbb{H}(J_1(mx))) =\emptyset,$ implying  $\epsilon(\Bbb{H}(J_1(mx))) = 0.$ Thus the result holds.\\ \\
For $n = 2,$ the Jaco graph $J_2(mx)$ has the prime Jaconian vertex $v_1$ and $d^+(v_1) = 1.$ $\Bbb{H}(J_2(mx))$ is the null graph on vertex $v_2$ so $A(\Bbb{H}(J_2(a)))=\emptyset,$ implying  $\epsilon(\Bbb{H}(J_2(mx))) = 0.$ Thus the result holds.\\ \\
Now assume it holds for all vertices $v_i,$ $i\leq k-1.$ Thus vertex $v_k$ has attained its in-degree $d^-(v_k).$ To attain $d(v_k) = \Delta (J_n(mx)),$ exactly $\Delta (J_n(mx))-d^-(v_k) = d^+(v_k)$ arcs can be linked additionally. So the result holds for vertices $v_1, v_2, v_3,\cdots, v_k.$\\ \\
Clearly we also have $A(\Bbb{H}(J_n(mx)))\subset A(J_n(mx)).$ Hence, $\epsilon(J_n(mx))= \epsilon(\Bbb{H}(J_n(mx))) + d^+(v_1) + d^+(v_2) + d^+(v_3) +\cdots+ d^+(v_k).$ \\
Since it is known that $\epsilon(\Bbb{H}(J_n(mx))) = 1 + 2 + 3 +\cdots + (n - \Delta (J_n(mx))-1)$ the result can be written as  $\epsilon(J_n(mx)) =\frac{1}{2}(n-k)(n-k-1)+ \sum \limits_{i=1}^{k}d^+(v_i). $
\end{proof}
\begin{corollary}
The number of arcs of a Jaco graph $J_n(mx)$ having vertex $v_i$ as the prime Jaconian vertex, can also be expressed recursively as
\begin{equation*} 
\epsilon (J_{n+1}(mx)) =
\begin{cases}
\epsilon(J_n(mx)) - i + n &\text {if $d(v_i) = mi$}\\
\epsilon(J_n(mx)) - i + (n+1) & \text {if $d(v_i) < mi$}\\ 
\end{cases}
\end{equation*} 
\end{corollary}
\begin{proof}
Consider the Jaco graph $J_n(mx)$ having vertex $v_i$ as the Jaconian vertex.\\ \\
Case1:  If $ d(v_i) = mi$ and the vertex $v_{(n+1)}$ is added to construct $J_{(n+1)}(f(x))$ only the arcs 
$(v_{(i+1)}, v_{(n+1)}), .... (v_n, v_{(n+1)})$ can be linked additionally. This amounts to $(n - i)$ arcs.\\ \\
Case 2: If $ d(v_i) < mi$ and the vertex $v_{(n+1)}$ is added to construct $J_{(n+1)}(mx)$ only the arcs 
$(v_i, v_{(n+1)}), (v_{(i+1)}, v_{(n+1)}), .... (v_n, v_{(n+1)})$ can be linked additionally. This amounts to $(n - i +1)$ arcs. 
\end{proof}
\begin{lemma}(Also see [7].)
\noindent We find for $m=1$ and the series $(a_n)_{n \in \Bbb N_0}$ defined by
\begin{center} $a_0 = 0, a_1 = 1, a_{n \geq 2} = \min \{ k < n: k + a_k \geq n \},$
\end{center}
that Lemma 1.1 changes to: \\
\noindent (a) $d^+(v_n) + d^-(v_n) = n.$\\
(b) $d^-(v_{n+1}) \in \{ d^- (v_n),\, d^-(v_n) + 1 \}.$\\
(c) If $(v_i, v_k) \in A(J_\infty(x))$ and $i < j < k,$ then $(v_j, v_k) \in A(J_\infty(x)).$\\
(d) $d^+ (v_n) = a_n,$ $ n \geq 2.$
\end{lemma}
\begin{corollary}
Note that $(a))$ and $(c)$ above entail that $d^+(v_{n+1}) = n+1- d^-(v_{n+1}) \in \{n-d^-(v_n), n-d^-(v_n) +1\}$ and that $(d)$ then implies tat the series $(a_n)$ is well defined and ascending, more specifically, $a_{n+1} \in \{a_n, a_n+1\},$ $(n\in \Bbb N_0).$
\end{corollary}
\begin{lemma}
Let $i \in\Bbb N.$ Then $d^+(v_{i+d^+(v_i)}) = i = d^+(v_{i+d^+(v_{i+d^+(v_{i-1})})}).$
\end{lemma}
\begin{proof}
let $i+d^+(v_i) = k.$ Certainly, $i+d^+(v_i) \geq k,$ so $d^+(v_k) \leq i.$ From Lemma 3.4 it follows that $i-1+d^+(v_{i-1}) \leq i-1+ d^+(v_i) < i+d^+(v_i)$ so $d^+(v_k) \geq i.$ Let $\ell = i+ d^+(v_{i-1}).$ Since $d^+(v_i) \geq d^+(v_{i-1})$ we have, $d^+(v_\ell) \leq i$ and since, $i-1+ d^+(v_{i-1}) < \ell$ we have, $d^+(v_\ell) =i.$
\end{proof}
The next theorem with proof is repeated here. It appears in [6, 7] which are announcement papers only. Applicable reading related to Fibonacci numbers and Zeckendorf representations is found in [5, 9].
\begin{theorem}
(Bettina's Theorem)(Also see: [6, 7]). Let $\Bbb{F} = \{f_0, f_1,f_2, ...\}$ be the set of Fibonacci numbers and let $n=f_{i_1} + f_{i_2} + ... + f_{i_r}, n\in \Bbb N$ be the Zeckendorf representation of $n.$ Then 
\begin{center}
$d^+(v_n) = f_{i_1-1} + f_{i_2-1} + ... +f_{i_r-1}.$
\end{center}
\end{theorem}
\begin{proof}
Through induction we have that first of all, $1=f_2$ and $d^+(v_1) =1=f_1.$ Let $2 \leq n = f_{i_1} + f_{i_2} + ... + f_{i_r}$ and let $k = f_{i_1-1} + f_{i_2-1} + ... + f_{i_r-1}.$ If $i_r \geq 3,$ then $k = f_{i_1-1} + f_{i_2-1} + ... + f_{i_r-1}$ is the Zeckendorf representation of $k$, such that induction yields $d^+(v_k) = k = f_{i_1-2} + f_{i_2-2} + ... + f_{i_r-2}.$ Since $k + d^+(v_k)  = f_{i_1-1} + f_{i_1-2} + f_{i_2-1} + f{i_2-2} + ... f_{i_r-1} + f_{i_r-2} = f_{i_1} + f_{i_2} + ... f_{i_r} = n,$ read with Lemma 3.6 yields $d^+(v_n) = k.$ \\ \\
Finally consider $n=f_{i_1} + f_{i_2} + ... + f_{i_r}, i_r =2.$ Note that $n >1$ implies that $i_{r-1} \geq 4$ and that the Zeckendorf representation of $n-1$ given by $n-1 = f_{i_1} + f_{i_2} + ... + f_{i_{r-1}}.$ Let $k= d^+(v_{n-1}).$ Through induction we have that, $k = f_{i_1-1} + f_{i_2-1} + ... + f_{i_{r-1}-1},$ and since $i_{r-1} \leq 4,$ this is the Zeckendorf representation of $k.$ Accordingly, $d(v_k) = f_{i_1-2} + f_{i_2-2}+ ... + f_{i_{r-1}-2},$ and $k+ d^+(v_k) =  f_{i_1-1} + f_{i_1-2} + f_{i_2-1} + f{i_2-2} + ... f_{i_{r-1}-1} + f_{i_{r-1}-2} = n-1.$ It follows that $d^+(v_n) > k = d^+(v_{n-1}).$ From Corollary 3.5 it follows that $d^+(v_n) = k+1 = (f_{i_1-1} + f_{i_2-1} + ... + f_{i_{r-1}-1}) + f_1 = f_{i_1-1} + f_{i_2-1} + ... + f_{i_r-1}.$
\end{proof}
 Since $d^-(v_n) = n - d^+(v_n)$, the number of edges is also given by $\epsilon(J_n(1)) = \frac{1}{2}n(n+1) - \sum\limits^n_{i=1}d^+(v_i).$ Furthermore, for $n\geq 2$ we have $d^+(v_1) = 1$, so we rather consider $\epsilon (J_n(1)) =(\frac{1}{2}(n(n+1) - 1) - \sum\limits^n_{i=2}d^+(v_i).$ Bettina's theorem can be applied to the last term to determine the number of arcs of $J_n(x).$\\ \\
\textbf{Illustration 2.} For the Jaco Graph $J_{15}(x)$ we have;\\ $\epsilon (J_{15}(1)) = \frac{1}{2}.15.(15+1) -1 - \sum\limits^{15}_{i=2} = 119 - \sum\limits^{15}_{i=2}d^+(v_i).$\\ \\
\noindent Now, $f_2 = 1, f_3 = 2, f_4 = 3, f_4 + f_2 =4, f_5 = 5,f_5 + f_2 =6, f_5 + f_3=7, f_6 =8, f_6 + f_2=9, f_6 + f_3 = 10,  f_6 + f_4 = 11, f_6 +f_4 + f_2 = 12, f_7 = 13, f_7 + f_2 = 14,$ and $f_7 +f_3 = 15. $\\ \\
\noindent From Bettina's Theorem it follows that;\\ $\sum\limits^{15}_{i=2}d^+(v_i) = f_2 + f_3 + (f_3+f_1) + f_4 + (f_4+f_1) + (f_4 + f_2) + f_5 + (f_5 + f_1) + (f_5 + f_2) +(f_5 + f_3) + (f_5+ f_3+ f_1) +f_6 + (f_6 + f_1) + (f_6 + f_2) = 5f_1 + 4f_2 + 4f_3 + 3f_4 +5f_5 +3f_6 = f_1 +5f_3 +5f_5 +3f_7 = 75.$\\ \\
So, $\epsilon (J_{15}(1)) = 119 - 75 = 44.$\\ \\
Now we present a result for a special case in respect of $f(x) = mx + c.$
\begin{theorem}
For the Jaco graph $J_{f(f(1) + 1) + 1) +1}(f(x))$ we have that:\\ $\epsilon(J_{f(f(1) + 1) + 1) +1}(f(x))) = \epsilon(K_{f(1) + 1}) + \epsilon(K_{f(f(1)+1)) - f(1)}) + \frac{1}{2}mf(1)\cdot (f(1) - 1),$ alternatively,\\
$\epsilon(J_{m^2+m(c+1)+2}(f(x))) = \epsilon(K_{(m+c) + 1}) + \epsilon(K_{m^2+m(c-1) -c+2}) + \frac{1}{2}m(m+c)\cdot ((m+c) - 1).$ 
\end{theorem}
\begin{proof}
Clearly the Jaco graph $J_{f(1) + 1}(f(x))$ is the complete graph $K_{f(1) + 1} = K_{m+c+1},$ so the first term including the first alternative term, follows. Clearly the induced subgraph $\langle v_i: f(1) + 1 \leq i \leq f(f(1) + 1) + 1\rangle$ is the complete graph $K_{(f(f(1) + 1) + 1) - (f(1) + 1))} = K_{m^2 + m(c-1) - c +2},$ so the second term including the second alternative term, follows.\\\\
Consider the graph $J'$ to be the graph obtained by appending the two complete graphs above at the common vertex $v_{f(1) + 1}.$ To obtain the Jaco graph $J_{f(f(1) + 1) + 1) +1}(f(x))$ from graph $J'$ we have to add all defined arcs from vertices $v_j \in \{v_i: 2 \leq i \leq f(1)\}$ to vertices $v_k \in \{v_\ell: f(1) +2 \leq \ell \leq f (f(1) + 1)\}.$ Choose any such vertex $v_j$ and we have $d_{J'}(v_j) =f(1)= m+c.$ Relabel the vertex $v_j$ to carry index $j-1$ and denote the vertex $v^*_{j-1}.$ So now we have that $d_{J'}(v^*_{j-1}) = m+c.$ Since $d_{J_{f(f(1) + 1) + 1) +1}(f(x))}(v_j) = mj +c,$ the additional $mj +c - (m+c) = m(j-1)$ arcs must be added.\\\\
Hence, $\sum\limits_{i=2-1}^{m+c-1}mi = m\cdot \sum\limits_{i=2-1}^{m+c-1}i$ arcs are required. This simplifies to $m\cdot \sum\limits_{i=1}^{m+c-1}i = m\cdot(\frac{1}{2}(m+c)(m+c-1)$ arcs. By this the result is settled.
\end{proof}
\subsection{Linear Function Corresponding to a Linear Jaco Graph}
It is noted that if for a sufficiently large linear Jaco Graph $J_n(f(x))$ we have two vertices $v_i, v_j$ for which $d(v_i) = f(i)$ and $d(v_j)= f(j)$ then the linear function $f(x)$ can be derived by solving the simultaneous equations:
\begin{center}
$mi + c = d(v_i)$\\ $mj + c = d(v_j).$
\end{center}
The smallest linear Jaco graph for which this is possible is for $J_{f(2)+1}(f(x))$ hence, knowing that $d(v_2) = f(2)$ in the given linear Jaco graph.
\begin{proposition}
If for a linear Jaco graph we have that $d(v_i) = f(i)$ and $d(v_{i+1}) = f(i+1)$ then for maximum $i', j'$ for which the arcs $(v_i, v_{i'}), (v_{i+1}, v_{j'})$ exist, we have $j' - i' \in \{m, m+1\}.$
\end{proposition}
\begin{proof}
Let $d(v_i) = f(i)$ and $d(v_{i+1}) = f(i+1)$ in a sufficiently large linear Jaco graph $J_n(f(x)).$ Since $d(v_{i+1}) - d(v_i) = m(i+1) + c - mi -c = m$ and from Lemma 1.1(b) we have $d^-(v_{i+1}) \in \{d^-(v_i), d^-(v_i) +1\}$, the result follows.
\end{proof}
Each positive integer $k$ can be written as $k+1$ sums of non-negative integers $m+c$, $m\geq 0,c \geq 0.$ If in definition 1.1 we relax the lower  limit on $m$ and allow $m \geq 0$ we say for a given $k$ that the linear Jaco graphs corresponding to the functions $f_i(x) = m_ix + c_i$, $m_i + c_i = k$ and $ 1 \leq i \leq k+1,$ are $f$-related Jaco graphs.\\\\
For $m=0$ and $c \geq 0$ we have two special classes of disconnected linear Jaco graphs. For $c= 0$ the Jaco graph $J_n(0)$  is a null graph (\emph{edgeless graph}) on $n$ vertices. For $c= k > 0$, the Jaco graph $J_n(k) = \bigcup\limits_{\lfloor\frac{n}{k+1}\rfloor-copies}K_{k+1} \bigcup K_{n-(k+1)\cdot \lfloor\frac{n}{k+1}\rfloor}.$\\\\
\textbf{Illustration 3.} Figure 2 depicts the linear Jaco graph $J_{15}(3)= K_4\bigcup K_4 \bigcup K_4 \bigcup K_3.$\\\\\\\\\\\

\begin{figure}
\centering
\includegraphics[width=0.7\linewidth]{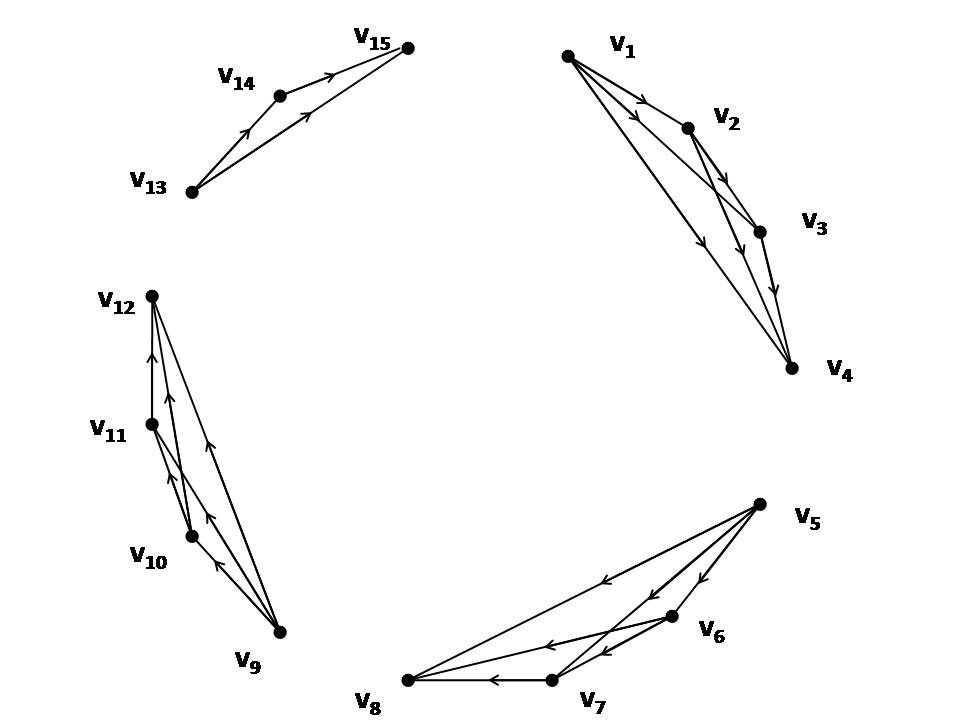}
\caption{}
\label{fig:Fig2}
\end{figure}

We note that a complete graph $K_n$ is a linear Jaco graph to any linear function $f(x) = mx +c$ if $n \leq m+c + 1.$ And, if $K_n$ is a $f$-related linear Jaco graph, the complete graph corresponds to any of the $m+c+1$ defined linear functions. It implies that we need at least a two-component subgraph of complete graphs say $K_s \bigcup K_t$ to derive the unique linear function $f(x) = k-1,$ $k = max\{s,t\}.$
\section{Conclusion}
Thus far all studies of Jaco graphs were indeed on linear Jaco graphs. We largely generalised the results announced in [6, 7]. Also note that the Jaco graphs $J_n(1)$ and $J_n(a)$ as discussed in [6, 7] has a slightly different meaning to that defined in this paper. This might make comparative reading somewhat confusing. The generalisation suggests uniform notation. Amongst the many open problems that exist within the family of Jaco graphs the authors view, finding a closed formula for the number of arcs (edges) of a Jaco graph the most challenging and interesting. Perhaps Theorem 3.8 lays the foundation to further seek such closed formula.\\\\
It is proposed that the generalisation of the concepts to polynomial Jaco graphs $J_n(f(x))$ where $f(x) = \sum\limits_{i=1}^{t}a_ix^i + c,$ $ a_i,x \in \Bbb N$ and $c \in \Bbb N_0,$ will be worthy to study. Definition 1.1 will have to change to define the orientation of edges resulting from $f(x) < 0.$ Also, given a sufficiently large polynomial Jaco graph, deriving the polynomial function can be formalised.\\\\
References \\ \\
$[1]$ C. Ahlbach, j. Usatine, N. Pippenger, \emph{Efficient Algorithms for Zerckendorf Arithmetic,} Fibonacci Quarterly, Vol 51, No. 13, (2013): pp 249-255. \\
$[2]$  J.A. Bondy and U.S.R. Murty, \textbf{Graph Theory with Applications,} Macmillan Press, London, (1976). \\
$[3]$ G. Chartrand and L. Lesniak, \textbf{Graphs and Digraphs}, CRC Press, 2000.\\
$[4]$ J.T. Gross and J. Yellen, \textbf{Graph Theory and its Applications}, CRC Press, 2006.\\
$[5]$ D. Kalman and R.Mena, \emph{The Fibonacci Number - Exposed,} Mathematics Magazine, Vol 76, No. 3, (2003): pp 167-181. \\ 
$[6]$ J. Kok, P. Fisher, B. Wilkens, M. Mabula, V. Mukungunugwa, \emph{Characteristics of Finite Jaco Graphs, $J_n(1), n \in \Bbb N$}, arXiv: 1404.0484v1 [math.CO], 2 April 2014. \\ 
$[7]$ J. Kok, P. Fisher, B. Wilkens, M. Mabula, V. Mukungunugwa,\emph{Characteristics of Jaco Graphs, $J_\infty(a), a \in \Bbb N$}, arXiv: 1404.1714v1 [math.CO], 7 April 2014. \\ 
$[8]$ D.B. West, \textbf{Introduction to Graph Theory}, Pearson Education Incorporated, 2001.
$[9]$ E. Zeckendorf, \emph {Repr\'esentation des nombres naturels par une somme de nombres de Fibonacci ou de nombres de Lucas,} Bulletin de la Soci\'et\'e Royale des Sciences de Li\`ege, Vol 41, (1972): pp 179-182. \\ \\
\noindent {\scriptsize [Remark: The concept of Jaco Graphs followed from a dream during the night of 10/11 January 2013 which was the first dream Kokkie (first author) could recall about his daddy after his daddy passed away in the peaceful morning hours of 24 May 2012, shortly before the awakening of Bob Dylan, celebrating Dylan's 71st birthday]}
\end{document}